\documentclass[a4paper]{article}

\usepackage{amsmath}
\usepackage{amsfonts,amssymb,amsthm}
\usepackage[final]{graphicx}

\newtheorem{thm}{Theorem}
\newtheorem{corr}{Corollary}
\newtheorem{lem}{Lemma}

\newtheorem{rem}{Remark}
\newtheorem{df}{Definition}

\newcommand{\RM}{\mathbb{R}}

\newcommand{\CM}{\mathbb{C}}
\newcommand{\MM}{\,\mbox{\bf M}}
\newcommand{\WM}{\,\mbox{\bf W}}

\newcommand{\HM}{\,\mbox{\bf H}}

\newcommand{\cn}{\operatorname{cn}}
\newcommand{\dn}{\operatorname{dn}}

\title{The Transverse Instability of Periodic Waves in Zakharov-Kuznetsov Type Equations}

\author{Mathew A. Johnson\footnote{
The author would like to thank Jared C. Bronski for suggesting this problem and for his many
useful conversations during the preparation and revision of this manuscript.
Also, the author gratefully acknowledges
support from a National Science Foundation Postdoctoral Fellowship under
grant DMS-0902192.}}

\theoremstyle{notation}

\begin{document}
\bibliographystyle{plain}

\maketitle
\begin{center}
Department of Mathematics\\
Indiana University\\
831 East Third Street\\
Bloomington, IN 47405 USA\\
matjohn@indiana.edu
\end{center}

\begin{center}
\end{center}

\begin{abstract}
In this paper, we investigate the instability of one-dimensionally stable
periodic traveling wave solutions of the
generalized Korteweg-de Vries equation to long wavelength transverse perturbations
in the generalized Zakharov-Kuznetsov equation
in two space dimensions.
By deriving appropriate asymptotic expansions of the periodic Evans function, we
derive an index which yields sufficient conditions for transverse instabilities
to occur.  This index is geometric in nature, and applies to any periodic traveling wave profile under
some minor smoothness assumptions on the nonlinearity.  We also describe the analogous theory for
periodic traveling waves of the generalized Benjamin-Bona-Mahony equation to long wavelength
transverse perturbations in the gBBM-Zakharov-Kuznetsov equation.
\end{abstract}


\section{Introduction}
It is well known that the Korteweg-de Vries (KdV) equation
\begin{equation*}
u_t=u_{xxx}+u u_x, \;\;(x,t)\in\RM^2
\end{equation*}
arises as an model for one-dimensional long wavelength surface waves propagating in weakly nonlinear dispersive media,
as well as the evolution of weakly nonlinear ion acoustic waves in plasmas [23].
There are several weakly two-dimensional variations on the KdV equation.  Of particular importance is
the
Zakharov-Kuznetsov (ZK) equation [26]
\[
u_t=\left(u_{xx}+u_{yy}\right)_x+uu_x,
\]
which arises when in the study of geophysical fluid dynamics in anisotropic settings and ionic-acoustic waves
in magnetized plasmas. 

In many applications, it is useful to introduce a parameter $\alpha$ in the nonlinearity
as follows:
\[
u_t=u_{xxx}+\alpha u u_x.
\]
Notice that by scaling one can always assume that $\alpha$ has a fixed sign.
In the context of plasmas, there are situations in which the physical parameters in the problem (such as temperature
and density) force $\alpha$ to vanish.  In neighborhood of such parameter values, it is necessary
to consider higher terms in the nonlinearity and one usually resorts to using a modified KdV (mKdV) equation
\[
u_t=u_{xxx}+\beta u^2u_x
\]
to describe such ion acoustic waves.  Here, the sign of $\beta$ can not be forced to be definite by
scaling considerations and the two signs of $\beta$ correspond to different physical phenomenon.  For
example, the focusing mKdV equation (corresponding to $\beta>0$) can be derived as a model for the evolution
of ion acoustic perturbations with a negative ion component, while the defocusing mKdV (corresponding to $\beta<0$)
can be derived as a model for the evolution of ion acoustic perturbations in a plasma with two
negative ion components of different temperature.  Moreover, the structure of the exponentially decaying solutions
are radically different: for $\beta>0$ the solitary
waves are homoclinic to a constant state, while for $\beta<0$ the solitary waves correspond to shock like solutions
which are heteroclinic to opposite constant states.  As in the case of the KdV equation, the modified ZK (mZK) equation
\[
u_t=\left(u_{xx}+u_{yy}\right)_x+\beta u^2 u_x
\]
occur naturally as weakly two-dimensional variations of the mKdV.

In order to encompass as many physical
applications as possible, the focus of our study will be on 
the generalized KdV (gKdV) equation
\begin{equation}
u_t=u_{xxx}+f(u)_x\label{gkdv}
\end{equation}
along with the generalized ZK (gZK) equation
\begin{equation}
u_t=\left(u_{xx}+u_{yy}\right)_x+f(u)_x\label{gZK}
\end{equation}
For suitable nonlinearities (satisfying mild convexity and smoothness assumptions), the gKdV equation admits
asymptotically constant solutions known as the solitary waves.  It seems natural to question the stability of such 
solutions of the gKdV as one-dimensional solutions in the gZK equations.  It is clear
the no sense of stability can occur unless the solitary wave is itself stable as a solution
of the corresponding $y$-independent equations.  
It is well known (see [5], [7], and [20]) that such planar solitary waves are one-dimensionally (orbitally)
stable to localized perturbations as solutions of the gKdV  if
\[
\frac{\partial}{\partial c} \int_\RM u^2(x+ct;c)dx
\]
is positive, and are (exponentially) unstable to such perturbations if it is negative: here $u$ is the profile of the
solitary wave and $c$ is the wave speed.  Physically, this says that such a solution is stable precisely
when its momentum is an increasing function of the wave speed.  In the case of power-law nonlinearities
$f(u)=u^{p+1}$, $p\geq 1$, this is equivalent with stability if $1\leq p<4$ and instability if $p>4$.

The transverse instability of one-dimensionally stable solutions of the gKdV in the gZK equation has been well
studied in the solitary wave context.  In the case of the ZK equation, Bridges [9] derived a geometric
condition for the long wavelength transverse instability of solitary solutions of \eqref{gkdv} by using the multi-symplectic
structure of the ZK.  There, the author derived an index which could be viewed as the Jacobian of
a particular map and whose sign determined the transverse stability of the underlying wave.  Moreover,
Bridges was able to evaluate this Jacobian explicitly for power-law nonlinearities $f(u)=u^{p+1}/(p+1)$
by using the fact that all solitary wave solutions in this case can be expressed in terms of hyperbolic functions.

If one rather considers the periodic solutions of \eqref{gkdv}, then the stability theory
is much less understood.  This reflects the fact that the spectrum of the corresponding
linearized operators is purely continuous, and hence it seems 
more difficult for nonlinear periodic waves of the gKdV to be stable than their solitary wave
counterparts.  Moreover, the periodic waves of the gKdV in general have a much more
rich structure than the solitary waves: even in the case of power-law nonlinearities it is not possible to
write down a general elementary representative for all periodic solutions of the gKdV, which stands in contrast to the solitary wave theory.
It should be noted, however, that the periodic solutions are perhaps
more physically relevant as they are used to model wave trains: such patterns are prevalent
in a variety of applications.  Up to this point, there seems to be two types of stability
results in the periodic case: nonlinear stability to perturbations with periods which are integer multiples of the underlying
waves period (see [1], [2], [8], [12], and [19]), and spectral stability to localized perturbations
(see [10], [14], and [16]).  In general, the methods employed in the nonlinear stability analysis parallels that
of the energy functional framework from the solitary wave theory [15].  However, such 
methods do not extend in a straight forward manner if one considers localized perturbations.  Instead, the techniques in this case mainly center
around studying the spectrum of the corresponding linearized operators and applying
regular perturbation theory: see for example [16].
Recently [10], there has been progress with analyzing the linearized spectrum in a neighborhood of the
origin in the spectral plane by using periodic Evans function methods introduced by Gardner [13].
Such unstable spectrum corresponds to instabilities to slow modulations of the underlying wave, i.e.
to long wavelength perturbations.  The authors in [10] essentially conduct a rigorous
modulation theory calculation to derive asymptotic expansions of the Evans function near the origin.

In this paper, we will extend the methods in [10] in order to analyze the spectral stability
of $y$-independent solutions of the gZK to long wavelength transverse perturbations.
In particular, we will conduct a rigorous modulation theory calculation.  To this end
we will first show the corresponding linearized spectral problem has a periodic
eigenvalue at the origin of multiplicity greater than one.  Using the integrability of the ODE defining the traveling wave profiles
of \eqref{gZK} and analyzing the corresponding periodic Evans functions,
we are able to analytically determine the leading order behavior of the periodic eigenvalues
bifurcating from the origin in the wavenumber of the transverse perturbation.  This information yields sufficient
conditions for the spectrum of the perturbed spectral problems to have unstable spectrum in a neighborhood of the origin.
These conditions are geometric in nature and can be expressed as the Jacobian of a particular map from the traveling wave parameters to the
period and mass of the underlying wave.  The positivity of this index implies modulational instability of the underlying periodic
wave to long wavelength transverse perturbations.  By using the results of [10], [11], and [17] we are able to evaluate
the sign of this Jacobian in several cases involving power-law nonlinearities: in particular, we are able to make
conclusive instability statements for the ZK and focusing mZK equations.

Also of interest in this paper is the analogous periodic transverse instability theory for the generalized
Benjamin-Bona-Mahony (gBBM) equation
\begin{equation}
u_t-u_{xxt}+u_x+f(u)_x=0.\label{gbbm}
\end{equation}
When $f(u)=u^2$, equation \eqref{gbbm} is precisely the Benjamin-Bona-Mahony equation, also known as the regularized long wave equation,
which arises as an alternate model to the gKdV equation as a description of gravity water waves in the long wave regime (see [6] and [21]).
As with the gKdV equation, a natural weakly two-dimensional generalization of \eqref{gbbm} is the gBBM-ZK equation [25]
\begin{equation}
u_t-\left(u_{xt}+u_{yy}\right)_x+u_x+f(u)_x=0.\label{gbbmzk}
\end{equation}
Recently, the stability of periodic traveling waves of the gBBM equation to both periodic
and localized perturbations has been conducted [18] using the Evans function techniques of [10].  In particular, it was
found that the stability theory closely resembles that of the gKdV equation.  As to be expected then, if one considers
the transverse instability of such a solution of \eqref{gbbm} to long wavelength transverse perturbations in the
gBBM-ZK equation the resulting theory parallels that of the gKdV equation described above.  
Therefore, we will only outline the corresponding transverse instability theory in this case.

The outline of this paper is as follows.  In section 2 we recall from [10] and [17]
the basic properties of the periodic traveling waves solutions
of \eqref{gkdv}.  In particular, we will discuss a parametrization of such solutions which is useful in our calculations.
We then carry out our transverse instability analysis for the gKdV equation in section 3.
In section 4, we outline the analogous theory for the gBBM equation and end with some
discussion and closing remarks in section 5.

\section{Properties of the Periodic Traveling Waves of gKdV}

In this section, we recall the basic properties of the periodic traveling wave solutions of \eqref{gkdv}.
A traveling wave solution of the gKdV with positive wave speed $c>0$
is a solution of the traveling wave ordinary differential equation
\begin{equation}
u_{xxx}+f(u)_x-cu_x=0,\label{travelode}
\end{equation}
i.e. they are solutions of \eqref{gkdv} which are stationary in the moving coordinate frame
defined by $x+ct$.  Clearly, such solutions are reducible to quadrature and satisfy
\begin{align}
u_{xx}+f(u)-cu&=a,\label{quad1}\\
\frac{1}{2}u_x^2+F(u)-\frac{c}{2}u^2-a u&=E,\label{quad2}
\end{align}
where $a$ and $E$ are real constants of integration, and $F$ satisfies $F'=f$, $F(0)=0$.
In order to ensure the existence of periodic orbits of \eqref{travelode}, we must require that the effective
potential
\[
V(u;a,c)=F(u)-\frac{c}{2}u^2-au
\]
has a non-degenerate local minimum.  Notice this places a restriction on the admissible class of nonlinearities
as well as the allowable parameter regime for a given nonlinearity.  This motivates the following definition.

\begin{df}
We define $\Omega\subset\RM^3$ to be the open set consisting of all triples $(a,E,c)$
such that \eqref{travelode} has at least one periodic solution.
\end{df}

When $(a,E,c)\in\Omega$, we use the notation $u(x;a,E,c)$ to denote a particular periodic solution\footnote{In the case
where more than one such solution exists for a particular $(a,E,c)$, we can distinguish them by their
initial values.}.  
Given an $(a,E,c)\in\Omega$, we will always assume that the roots $u_{\pm}$ of $E=V(u;a,c)$
which correspond to the absolute max and min, respectively, of our periodic wave
are simple.  It follows that $u_{\pm}$ are $C^1$ functions of $a$, $E$, and $c$ on $\Omega$ and,
with out loss of generality, $u(0)=u_{-}$.

\begin{rem}
Taking into account the translation invariance of \eqref{gkdv}, it follows that for each
$(a,E,c)\in\Omega$ we can construct a one-parameter family of periodic traveling wave
solutions of \eqref{gkdv}: namely
\[
u_{\xi}(x,t)=u(x+ct+\xi;a,E,c)
\]
where $\xi\in\RM$.  Thus, the periodic traveling waves of \eqref{gkdv} constitute a four
dimensional manifold of solutions.  However, this added constant of integration does not play
an important role in our theory and can be modded out.
\end{rem}

Given an $(a,E,c)\in\Omega$, notice that the period of the corresponding solution
can be represented as
\begin{equation}
T=T(a,E,c):= 2\int_{ u_{-}}^{u_{+}} \frac{du}{\sqrt{2\left(E-V(u;a,c)\right)}}.\label{period}
\end{equation}
The above integrand can be regularized at the square root branch
points $ u_-,\; u_{+}$ by a standard procedure (see [10] or [17]).
In particular, we can differentiate the above relation with respect to the parameters $a$, $E$, and $c$ within
the parameter regime $\Omega$.  Similarly the mass, momentum, and Hamiltonian of the traveling wave
are given by the first and second moments of this density, i.e.
\begin{align}
M(a,E,c) & = \int_0^T u(x)\; dx = 2\int_{u_-}^{u_+} \frac{u\; du}{\sqrt{2\left(E-V(u;a,c)\right)}}\label{mass} \\
P(a,E,c) & = \int_0^T u^2(x)\; dx = 2\int_{u_-}^{u_+} \frac{u^2\; du}{\sqrt{2\left(E-V(u;a,c)\right)}}\label{momentum} \\
H(a,E,c) &=  \int_0^T\left( \frac{u_x^2}{2} - F(u)\right)(x)~dx =  2\int_{u_-}^{u_+} \frac{E-V(u;a,c) - F(u)}{\sqrt{2\left(E-V(u;a,c)\right)}}\;du\label{energy} .
\end{align}
These integrals can be regularized as above, and represent conserved
quantities of the gKdV flow restricted to the manifold of periodic traveling
wave solutions. In particular one can differentiate the above expressions with respect to the parameters $(a,E,c)$.

\begin{rem}
Notice that in the derivation of the gKdV [3], the solution $u$ can represent either the horizontal velocity of a wave profile,
or the density of the wave.  Thus, the functional $M$ can properly be interpreted as a ``mass" since it is the integral of
the density over space.  Similarly, the functional $P$ can be interpreted as a ``momentum" since it is the integral
of the density times velocity over space.
\end{rem}

Throughout this paper, a large role will be played by the gradients of the above
conserved quantities.  Thankfully, the underlying classical dynamics of the \eqref{travelode}
provides many useful relations between these gradients.  To see this, note that
the classical action (in the sense of action angle variables) is given by
\begin{equation}
K(a,E,c) = \oint u_x \;du = \int_0^T u_x^2 \;dx = 2\int_{u_-}^{u_+}\sqrt{2(E-V(u;a,c))}\;du. \label{action}
\end{equation}
The derivative of this map as a function $K:\Omega\to\RM$ is given explicitly by
\[
\nabla_{a,E,c}K(a,E,c)=\left(M(a,E,c),T(a,E,c),\frac{1}{2}P(a,E,c)\right),
\]
where $\nabla_{a,E,c}=\left<\frac{\partial}{\partial a},\frac{\partial}{\partial E},\frac{\partial}{\partial c}\right>$.
This is easily verified by differentiating \eqref{action} and recalling that $E-V(u_{\pm};a,c)=0$.
Using the fact that $T,\;M,\;P,$ and $H$ are $C^1$ functions of parameters $(a,E,c)$ on
the domain $\Omega$, we see that
\begin{equation}
E \;\nabla_{a,E,c}( T) + a\; \nabla_{a,E,c} (M) + \frac{c}{2} \;\nabla_{a,E,c}( P) + \nabla_{a,E,c}( H) = 0:\label{overdetermined}
\end{equation}
see [10] for details.
Thus, although the subsequent theory is developed most naturally
in terms of the quantities $T$, $M$, and $P$, it is possible to restate our results
in terms of $M$, $P$, and $H$ so long as $E\neq 0$: this is desirable since these have a natural interpretation as conserved
quantities of the partial differential equation \eqref{gkdv}.

We now discuss the aforementioned  parametrization of the periodic solutions of \eqref{travelode} in more
detail.  A major technical necessity throughout this paper, especially in the analysis
of the gZK equation, is that the constants
of motion for the PDE flow defined by \eqref{gkdv} provide (at least locally) a good
parametrization for the periodic traveling wave solutions.  In particular, we assume
for a given $(a,E,c)\in\Omega$ that the conserved quantities
$(H,M,P)$ are good local coordinates for the periodic traveling waves near $(a,E,c)$ in the sense
that the map $(a,E,c)\to(H(a,E,c),M(a,E,c),P(a,E,c))$ has a unique
$C^1$ inverse in a neighborhood of $(a,E,c)$.  Adopting the notation
\[
\{f,g,h\}_{x,y,z} := \frac{\partial(f,g,h)}{\partial(x,y,z)}
\]
for $3 \times 3$ Jacobians, 
we see that this is possible exactly when $\{H,M,P\}_{a,E,c}$ is non-zero, which is equivalent to
$\{T,M,P\}_{a,E,c}$ being non-zero if $E\neq 0$ by \eqref{overdetermined}.
In particular, it was shown in [10] that 
the vanishing of $\{T,M,P\}_{a,E,c}$ forces a change in the structure the generalized periodic
null-space of the linearized operator: see Lemma \ref{perkdv:lem} below for more details.
This non-degeneracy condition has been shown to
generically hold in several cases of power-law nonlinearities in [11].
In the case of the gZK equation, the sign of this Jacobian will play an important role in determining of the
corresponding periodic traveling wave solution of the gKdV is a one-dimensionally stable solution
of the gZK equation.  When considering the transverse instability of such a solution, another
important Jacobian which arises is $\{T,M\}_{a,E}$, where this 2x2 Jacobian is defined
analogous to the above 3x3 case.  Notice that the non-vanishing of this Jacobian implies
the period and mass can be used to parameterize nearby periodic waves with a fixed wave speed.
It is precisely the sign of this Jacobian that determines the stability of such a solution to long wavelength
transverse instabilities in both the gZK and gKP equations.

It should be noted that the geometric quantities $\{T,M\}_{a,E}$ and $\{T,M,P\}_{a,E,c}$ also arise
naturally when considering the stability of a periodic traveling wave solution of the gKdV equation
perturbations whose period is an integer multiple of that of the underlying wave (see [10], [11], and [17]).
Since, as mentioned in the introduction, the long wavelength transverse instabilities
occur from the bifurcation of periodic eigenvalues from the origin
into the unstable space, it is not surprising that these quantities play a large role in the
transverse instability analysis.


\section{Transverse Stability Analysis}

In this section, we begin our transverse stability analysis for one-dimensionally stable
periodic traveling wave solutions of \eqref{gkdv}.
As mentioned before, a $T$-periodic traveling wave
solution $u(x,t)=u(x+ct)$ of \eqref{gkdv} corresponds to a $y$-independent $T$-periodic traveling wave solution solution of traveling
gZK equation
\begin{equation}
u_t=\left(u_{xx}+u_{yy}\right)_x+f(u)_x-cu_x.\label{travelgzk}
\end{equation}
We now consider a small localized perturbation of $u$ of the form
\[
\psi(x,y,t)=u(x)+\varepsilon v(x,y,t)+\mathcal{O}(\varepsilon^2)
\]
where $v(x,\cdot,t)\in L_{\rm}^\infty(\RM)$ for each $(x,t)\in\RM^2$ and $v(\cdot,y,t)\in L^2(\RM)$ for each $(y,t)\in\RM^2$,
and $|\varepsilon|\ll 1$ is a small parameter.  Forcing $\psi$ to solve \eqref{travelgzk} yields a hierarchy of
consistency conditions: the $\mathcal{O}(\varepsilon^0)$ equation holds since $u$ is a solution of \eqref{travelgzk}, and
the $\mathcal{O}(\varepsilon)$ equation reads
\begin{equation*}
-v_t=\partial_x\left(\mathcal{L}[u]-\partial_y^2\right)v
\end{equation*}
where $\mathcal{L}[u]=-\partial_x^2-f(u)+c$.
As this linearized equation is autonomous in time and in the spatial variable $y$, we may seek separated solutions
of the form\footnote{Alternatively, you could take the Laplace transform in time and the Fourier transform in the spatial variable $y$.}
\[
v(x,y,t)=e^{-\mu t+iky}v(x)
\]
where $\mu\in\CM$, and $k\in\RM$ is the transverse wavenumber of the perturbation.  This leads to the (ODE) spectral problem
\begin{equation}
\partial_x\left(\mathcal{L}[u]+k^2\right)v=\mu v\label{lineargzk}
\end{equation}
considered on the real Hilbert space $L^2(\RM)$.  Notice that we consider $\mathcal{L}[u]$ as a closed, self-adjoint
second order differential operator with periodic coefficients with dense domain in $L^2(\RM)$.

Our general strategy in analyzing \eqref{lineargzk} is as follows.
As the spectral problem involves differential operators with periodic coefficients,
standard results in Floquet theory imply that the $L^2$ spectrum is purely continuous, and consists
entirely of $L^\infty$ eigenvalues\footnote{Notice this shows one of the main difficulties in considering the
spectral stability of periodic solutions as opposed to the asymptotically constant case.  In the latter,
the continuous spectrum can in general be characterized via a Weyl sequence argument and the Fourier transform and shown to not contribute
to an instability.  Thus, one must only study the $L^2$ point spectrum of the linear operators.  In the periodic
context, however, the spectrum is entirely essential and hence any spectral instability comes from the continuous spectrum.}.
We define the monodromy matrix $\MM(\mu,k)$ to be the solution of the corresponding first order system
\begin{equation}
Y_x=\HM(x,\mu,k)Y,\;\;Y(0)={\bf I},\label{system}
\end{equation}
where ${\bf I}$ is the three by three identity matrix.  Following [13] we can then characterize the $L^\infty$ eigenvalues
as the roots of the periodic Evans function, which we define as
\[
D(\mu,k,\lambda):=\det\left(\MM(\mu,k)-\lambda{\bf I}\right).
\]
The complex constant $\lambda$ referred to as the Floquet multiplier and is related to the class of admissible
perturbations considered: for example, $\lambda=1$ corresponds to periodic perturbations with the same
period as the underlying wave.  It can be shown that, for a fixed transverse wavenumber $k$, $\mu$ is in the
spectrum of the corresponding linearized operator if and only if there exists a $\lambda\in\CM$ with $|\lambda|=1$
such that $D(\mu,k,\lambda)=0$, i.e. if and only if the corresponding monodromy operator has an eigenvalue on the unit
circle.  As a consequence, the $L^2$ spectrum of the corresponding spectral problems, for a fixed $k$,
can be parameterized by by the discrete zero-set of the functions $D(\mu,k,\lambda)$ with $|\lambda|=1$, i.e.
it is equal to the set
\[
\bigcup_{\kappa\in[-\pi,\pi)}\{\mu\in\CM:D(\mu,k,e^{i\kappa})=0\}.
\]
Moreover, it will be shown that the natural symmetries present in the gKP equations
forces $D(0,0,1)=0$, and hence the origin in the spectral plane is a periodic eigenvalue
for each of the spectral problems \eqref{lineargzk}.  Our goal
is to understand the nature of the spectrum in a neighborhood of $\mu=0$ when the transverse
wavenumber $k$ is very small,
and hence we must study the zero set of the function $D(\mu,k,1)$ in a neighborhood of
the point $(\mu,k,\lambda)=(0,0,1)$.  This will allow us to determine how these
periodic eigenvalues bifurcate from the origin in the transverse wavenumber and will lead us to a sufficient
condition for instability.




We begin our analysis by studying transverse instability within the gZK equation.
Throughout, we will work with a fixed periodic traveling wave solution $u(x;a,E,c)$ of \eqref{gkdv} with $(a,E,c)\in\Omega$
and consider the transverse instability of $u$ as a solution of \eqref{travelgzk}.  By linearizing \eqref{travelgzk}
around $u$ and applying the appropriate transforms in time and the transverse direction $y$ leads to the spectral
problem
\begin{equation*}
\partial_x\left(\mathcal{L}[u]+k^2\right)v=\mu v
\end{equation*}
considered on $L^2(\RM)$.  We begin by determining when a periodic traveling wave solution of the gKdV equation
is a one-dimensionally stable solution of the gZK equation, i.e. determining necessary
conditions which imply the $L^2(\RM)$ spectrum of the unperturbed linearized operator $\partial_x\mathcal{L}[u]$
is confined to the imaginary axis\footnote{It follows that by one-dimensionally stable, we mean
that it is a $y$-independent solution which is stable to $y$-independent perturbations.}.

First, we recall the necessary results of [10] where the spectrum
of the operator $\partial_x\mathcal{L}[u]$ was extensively analyzed in a neighborhood
of $\mu=0$.  The periodic traveling wave profile $u(\cdot;a,E,c)$
satisfies the traveling wave ODE \eqref{travelode}, and moreover satisfies the identities
\eqref{quad1} and \eqref{quad2}.  In particular, it follows from differentiating \eqref{quad1} that
the functions $u_x,$ $u_a$, and $u_E$ are linearly independent and satisfy 
\begin{align}
\mathcal{L}[u]u_x=0,\;\;\mathcal{L}[u]u_a=1,\;\;\mathcal{L}[u]u_E=0.\label{nullL:zk}
\end{align}
It follows that we can use these three functions to form a basis to construct the monodromy matrix $\MM(0,0)$ in this case.
In particular, it can be shown that the matrix $\MM(0,0)$ has $1$ as an eigenvalue of algebraic multiplicity
three and geometric multiplicity two.  Using this along with the Hamiltonian structure of the unperturbed
linear operator $\partial_x\mathcal{L}[u]$ and a bit of perturbation theory, we are able to prove
the following lemma.

\begin{lem}\label{perkdv:lem}
The function $D(\mu,k,1)$ satisfies the following properties.
\begin{enumerate}
\item $D(\mu,k,1)$ is an odd function of $\mu$ and an even function of $k$.
\item The limit $\lim_{\mu\to\infty}{\rm sign}\left(D(\mu,0,1)\right)$ exists and is negative.
\item The asymptotic relation
\[
D(\mu,0,1)=-\frac{1}{2}\{T,M,P\}_{a,E,c}\;\mu^3+\mathcal{O}(|\mu|^4)
\]
holds in a neighborhood of $\mu=0$.
\end{enumerate}
In particular, the corresponding periodic traveling wave solution of \eqref{gkdv} is
one-dimensionally unstable as a solution of \eqref{gZK} if $\{T,M,P\}_{a,E,c}<0$.
\end{lem}

\begin{proof}
The second and third parts have been proven in [10] using perturbation theory and asymptotic
analysis.  The fact that $D(\mu,k,1)$ is odd in $\mu$ was also proved in [10] using the fact
that the Hamiltonian structure of the operator $\partial_x\mathcal{L}[u]$
implies that $\MM(\mu,0)\sim\MM(-\mu,0)^{-1}$.  Similarly, $D(\mu,k,1)$ is even in $k$
since the perturbed spectral problem $\partial_x\left(\mathcal{L}[u]+k^2\right)v=\mu v$ is
invariant under the transformation $k\mapsto -k$.  
\end{proof}

We are now interested in the case where the transverse wavenumber $k$ is small.  By Lemma \ref{perkdv:lem},
$\mu=0$ is a periodic eigenvalue of the linear operator $\partial_x\mathcal{L}[u]$
of multiplicity three.  We now wish to see how these periodic eigenvalues
bifurcate from the origin in the transverse wavenumber $k$.  To do so, we determine the dominant
balance of the equation $D(\mu,k,1)=0$ for $\mu$ and $k$ small.  This is the content of the following lemma.

\begin{lem}\label{ZKEvans:lem}
The equation $D(\mu,k,1)=0$ has the following normal form in a neighborhood of $(\mu,k)=(0,0)$:
\[
-\frac{\mu^3}{2}\{T,M,P\}_{a,E,c}+\mu k^2\{T,M\}_{a,E}\int_0^Tu_x^2dx+\mathcal{O}(|\mu|^4+k^4)=0.
\]
\end{lem}

\begin{proof}
By standard arguments using a Newton diagram (see [4] or [24]), Lemma \ref{perkdv:lem} implies
the dominant balance of the equation $D(\mu,k,1)=0$ near $(\mu,k)=(0,0)$ must take the form
\[
-\frac{\mu^3}{2}\{T,M,P\}_{a,E,c}+ \frac{\mu k^2}{2}D_{\mu kk}(0,0,1)+\mathcal{O}(|\mu|^4+k^4)=0.
\]
Thus, it is enough to prove $D_{\mu kk}(0,0,1)=2\{T,M\}_{a,E}\int_0^Tu_x^2dx$.

To begin, 
we write \eqref{lineargzk} 
as a first order system of form \eqref{system} with
\[
\HM(x,\mu,k)=\left(
               \begin{array}{ccc}
                 0 & 1 & 0 \\
                 0 & 0 & 1 \\
                 -\mu-f''(u)u_x & c+k^2-f'(u) & 0 \\
               \end{array}
             \right).
\]
From \eqref{nullL:zk}, we can use the functions $u_x$, $u_a$, and $u_E$
to construct a corresponding matrix solution in the case $(\mu,k)=(0,0)$.
To this end, we define $\WM(x,\mu,k)$ to be the matrix solution which satisfies
\[
\WM(x,0,0)=\left(
             \begin{array}{ccc}
               u_x & u_a & u_{E} \\
               u_{xx} & u_{ax} & u_{Ex} \\
               u_{xxx} & u_{axx} & u_{Exx} \\
             \end{array}
           \right)
\]
and fix the initial condition $\WM(0,\mu,k)=\WM(0,0,0)$ for all $(\mu,k)\in\CM\times \RM$.
Defining 
$\delta\WM(\mu,k)=\WM(x,\mu,k)\big{|}_0^T$, it follows that
\[
D(\mu,k,1)=-\det\left(\delta\WM(\mu,k)\right)
\]
since $\det\left(\WM(0,0,0)\right)=-1$.  Thus, it suffices to analyze the matrix $\delta\WM(\mu,k)$.

Next, we compute $\delta\WM(0,0)$ with the plan of treating $\delta\WM(\mu,k)$ as a small
perturbation of this constant matrix.  Notice that by modding out the translation invariance, we can set
\begin{align*}
u(0)&=u_{-}=u(T)\\
u_x(0)&=0=u_{x}(T)\\
u_{xx}(0)&=-V'(u_{-})=u_{xx}(T).
\end{align*}
Using these equations along with the chain rule and \eqref{gkdv}, a straightforward calculation shows that
\[
\delta\WM(0,0)=\left(
                 \begin{array}{ccc}
                   0 & 0 & 0 \\
                   0 & V'(u_{-})T_a & V'(u_{-})T_E \\
                   0 & 0 & 0 \\
                 \end{array}
               \right).
\]
In order to calculate $D_{\mu kk}(0,0,1)$, we now treat $\WM(x,\mu,k)$ as a small perturbation
of $\WM(x,0,0)$ for $|(\mu,k)|_{\CM\times\RM}\ll 1$.  To this end, we notice that
\begin{equation}
\delta\WM(\mu,k)=\delta\WM(\mu,0)+k^2\WM_0(\mu,k)\label{eqn1}
\end{equation}
for some matrix $\WM_0(\mu,k)$ which is $\mathcal{O}(1)$ in $\mu$ and $k^2$.  Thus,
we must compute $\delta\WM(\mu,0)$ to first order and $\WM_0$ to leading order in $k$.
The computation of $\delta\WM(\mu,0)$ was carried out in detail in [10].  In particular,
it was proved that if we define the vector solutions corresponding to the functions
$u_x$, $u_a$, and $u_E$ at $(\mu,k)=(0,0)$ as $Y_1$, $Y_2$, and $Y_3$, respectively, then
\begin{align*}
\frac{\partial}{\partial\mu}Y_{1}(T,\mu,0)\big{|}_{\mu=0}&=
            \left(
              \begin{array}{c}
                0 \\
                V'(u_{-})\left(-T_c+u_{-}T_a-\frac{u_{-}^2}{2}T_E\right) \\
                0 \\
              \end{array}
            \right),\\
\frac{\partial}{\partial\mu}Y_{2}(T,\mu,0)\big{|}_{\mu=0}&=\left(
                                           \begin{array}{c}
                                             -\frac{d u_{-}}{da}\int_0^Tu_adx+\frac{du_{-}}{dE}\int_0^Tuu_adx \\
                                             * \\
                                             -\left(1+(c-f'(u_{-}))\frac{d u_{-}}{da}\right)\int_0^Tu_adx
                                                +\left(c-f'(u_{-})\right)\frac{du_{-}}{dE}\int_0^Tuu_adx \\
                                           \end{array}
                                         \right),
\end{align*}
and
\[
\frac{\partial}{\partial\mu}Y_{3}(T,\mu,0)\big{|}_{\mu=0}=\left(
                                           \begin{array}{c}
                                             -\frac{d u_{-}}{da}\int_0^Tu_Edx+\frac{du_{-}}{dE}\int_0^Tuu_Edx \\
                                             * \\
                                             -\left(1+(c-f'(u_{-}))\frac{d u_{-}}{da}\right)\int_0^Tu_Edx
                                                +\left(c-f'(u_{-})\right)\frac{du_{-}}{dE}\int_0^Tuu_Edx \\
                                           \end{array}
                                         \right),
\]
where $*$ represents terms which can be explicitly computed but are not necessary at this order in the perturbation calculation.
Thus,
\[
\delta\WM(\mu,0)=\delta\WM(0,0)+
\left(\sum_{j=1}^3\frac{\partial}{\partial\mu}Y_{j}(T,\mu,0)\big{|}_{\mu=0}\otimes e_j\right)\mu+\mathcal{O}(|\mu|^2)
\]
where $e_j$ is the $j^{th}$ column of the identity matrix on $\RM^3$, and $\otimes$ denotes the standard tensor product
of vectors in $\RM^3$.

Finally, we calculate $\WM_0(0,k)$ to first order in $k^2$.  To this end, notice that by the specific
forms of the first order variations of $Y_1$, $Y_2$, and $Y_3$ in $\mu$,
%
it is enough to compute the $k^2$ variation in the translation direction, i.e. we need only calculate
$\frac{\partial}{\partial k^2}Y_1(T,0,k)$ at $k=0$.  Variation of parameters yields
\begin{align*}
\frac{\partial}{\partial k^2}Y_1(T,0,k)\big{|}_{k=0}&=\WM(T,0,0)\int_0^T\WM(x,0,0)^{-1}\left(
                                                                                        \begin{array}{c}
                                                                                          0 \\
                                                                                          0 \\
                                                                                          u_{xx} \\
                                                                                        \end{array}
                                                                                      \right)dx
\\
&=\left(
    \begin{array}{c}
      \frac{\partial u_{-}}{\partial E}\int_0^Tu_x^2dx \\
      * \\
      (c-f'(u_{-}))\frac{\partial u_{-}}{\partial E}\int_0^Tu_x^2dx \\
    \end{array}
  \right)
\end{align*}
where the term $*$ is as before.  
Thus, we see that $\det\left(\delta\WM(\mu,k)\right)$ can be expanded as
\begin{equation}
\det\left(\delta\WM(0,0)+\left(\sum_{j=1}^3\frac{\partial}{\partial\mu}Y_{j}(T,\mu,0)\big{|}_{\mu=0}\otimes e_j\right)\mu
    +\left(\frac{\partial}{\partial k^2}Y_1(T,0,k)\big{|}_{k=0}\otimes e_1\right)k^2\right)+\mathcal{O}(|\mu|^4+k^4)\label{det1}
\end{equation}
in a neighborhood of $(\mu,k)=(0,0)$.  The matrix involved in the determinant in \eqref{det1}
can be simplified via a few elementary row operations.  In particular by replacing the third row with the
first row times the quantity $c-f'(u_{-})$ plus the third row, and replacing the first column with the first
column minus $u_{-}\mu$ times the first plus $u_{-}^2\mu/2$ times the third it follows that the determinant above
is equal to that of the matrix
\begin{align*}
\left(
  \begin{array}{ccc}
    k^2\frac{\partial u_{-}}{\partial E}\int_0^Tu_x^2dx  & \left(-\frac{d u_{-}}{da}\int_0^Tu_adx+\frac{du_{-}}{dE}\int_0^Tuu_adx\right)\mu &
                    \left(-\frac{d u_{-}}{da}\int_0^Tu_Edx+\frac{du_{-}}{dE}\int_0^Tuu_Edx\right)\mu \\
    -\mu V'(u_{-})T_c & V'(u_{-})T_a & V'(u_{-})T_E \\
    0 & -\mu\int_0^Tu_adx & -\mu\int_0^Tu_Edx \\
  \end{array}
\right).
\end{align*}
A straightforward computation now yields
\[
\frac{1}{2}\frac{\partial^3}{\partial\mu\partial k^2}\det\left(\delta\WM(\mu,k)\right)\big{|}_{(\mu,k)=(0,0)}=-\{T,M\}_{a,E}\int_0^Tu_x^2dx
\]
as claimed.
\end{proof}

From Lemma \ref{ZKEvans:lem}, we see that the dominant balance of the equation $D(\mu,k,1)=0$ for small $\mu$ and $k$
is a homogeneous polynomial of degree three in the variables $\mu$ and $k$.  In particular, the three
periodic eigenvalues at $\mu=0$ bifurcate analytically from the origin in the transverse wave-number $k$.
If we define $\mu=\alpha k+\mathcal{O}(k^2)$ for $|k|\ll 1$, it follows that $\alpha$ must be
a root of the polynomial
\[
Q(y):=-\frac{y^3}{2}\{T,M,P\}_{a,E,c}+\frac{y}{2}\{T,M\}_{a,E}\int_0^Tu_x^2dx.
\]
Our main theorem on the transverse instability of the periodic traveling wave solutions of the gKdV
in the ZK equation now follows easily.

\begin{thm}\label{transverse:ZK}
If $\{T,M,P\}_{a,E,c}$ is non-zero at $(a_0,E_0,c_0)\in\Omega$,
then the spectrally stable periodic traveling wave solution $u(x;a_0,E_0,c_0)$ of \eqref{gkdv}
is spectrally unstable to long wavelength transverse perturbations in the gZK equation \eqref{gZK}
if $\{T,M\}_{a,E}>0$ at $(a_0,E_0,c_0)$.
\end{thm}

\begin{proof}
By solving the equation $Q(y)=0$, Lemma \ref{ZKEvans:lem} implies that
there are three periodic eigenvalues in a neighborhood of the origin which are given by $\mu_0=o(k)$ and
\[
\mu_{\pm}=\pm|k|\sqrt{\frac{4\{T,M\}_{a,E}\int_0^Tu_x^2dx}{\{T,M,P\}_{a,E,c}}}+o(k).
\]
Since $u(x;a,E,c)$ was assumed to be a spectrally stable solution of \eqref{gkdv}, we know from Lemma \ref{perkdv:lem}
that $\{T,M,P\}_{a,E,c}>0$ and hence there will be two (non-zero) periodic
eigenvalues off the imaginary axis in the neighborhood of the origin if $\{T,M\}_{a,E}>0$.
\end{proof}

\begin{rem}\label{rem1}
Notice that if $\{T,M\}_{a,E}$ is negative, Theorem \ref{transverse:ZK} provides no information about the
transverse instabilities of periodic traveling wave solutions to the gKdV in the gZK equation.  This follows
from the fact that one eigenvalue is just $o(k)$, but we have no information about its confinement to the imaginary
axis.  Moreover, it is clear that Theorem \ref{transverse:ZK} can be extended to the case where $\{T,M,P\}_{a,E,c}=0$
since $D(\mu,k,1)$ is odd in $\mu$ and the leading order variation in $\mu$ must be negative by Lemma \ref{perkdv:lem}.
However, the vanishing of $\{T,M,P\}_{a,E,c}$ will force several branches of spectrum to bifurcate non-analytically from the origin.
\end{rem}

We now point out a few corollaries to Theorem \ref{transverse:ZK} in the case of dnoidal-like solutions
for power-law nonlinearities $f(u)=u^{p+1}$, $p\geq 1$.  By dnoidal-like solutions
we mean solutions which are bounded by a homoclinic or heteroclinic orbit in phase space, much like the
well known Jacobi $dn$ solutions of the focusing mKdV: see [12] for example.
To begin, we recall from [10] that the induced scaling in the wave-speed
yields the following asymptotic relations in a neighborhood
of a homoclinic orbit\footnote{Here, we use the notation $f(z)\sim g(z)$ as $z\to z_0$ to mean
that $\lim_{z\to z_0}f(z)/g(z)=1$ if $\lim_{z\to z_0}g(z)\neq 0$.}:
\begin{align}
\{T,M\}_{a,E}\sim -T_E M_a,\;\;\{T,M,P\}_{a,E,c}\sim-T_E M_a\left(\frac{2}{pc}-\frac{1}{2c}\right)P.\label{Jacobian-Limit}
\end{align}
If we now restrict to dnoidal-like solutions,
then, at least for $a$ sufficiently small, it is known by a result of Schaaf [22] that $T_E>0$ (see [10] or [17]).
Moreover, it was shown in [10] and [17] that $M_a<0$
for solutions of sufficiently large period, i.e. for orbits sufficiently close to the homoclinic orbit, and $|a|$ sufficiently small.
Thus, we immediately have the following corollary of Theorem \ref{transverse:ZK}.

\begin{corr}\label{transversecorr1:ZK}
In the case of a power-law nonlinearity, the dnoidal-like periodic traveling wave solutions of \eqref{gkdv} 
of sufficiently large period and $|a|$ sufficiently small are transversely unstable in the gZK equation \eqref{gZK}.  Moreover,
the periodic traveling wave solutions sufficiently close to a stationary solution in phase space
are also transversely unstable in the gZK equation.
\end{corr}

\begin{proof}
By Lemma \ref{perkdv:lem} and \eqref{Jacobian-Limit}, such a solution is unstable to long wavelength
transverse perturbations in the ZK equation if $p<4$.  Moreover, if $p>4$ then $\{T,M,P\}_{a,E,c}<0$
and hence the function $D(\mu,0,1)$ has a non-zero real root $\mu_*>0$.
By continuity the function $D(\mu,k,1)$ then, will have a root near $\mu_*$ with positive real part for $|k|\ll 1$.
Thus, although we are not guaranteed a long wavelength transverse instability in this case, we still
have spectral instability none the less.
\end{proof}

In the special case of the KdV equation ($f(u)=u^2$), we are able to use the integrability
of the traveling wave ODE \eqref{travelode} to show that $\{T,M\}_{a,E}>0$ 
for all $(a,E,c)\in\Omega$ (see [17]).  Moreover, up to Lie symmetries (Galilean invariance, spatial translation, etc.),
all periodic traveling wave solutions of the KdV equation can be expressed in terms of Jacobi elliptic functions as
\[
u(x,t)=6\alpha^2 \gamma^2 \cn^2(\alpha x+4\alpha^2(2\gamma^2-1)t;\gamma),
\]
where $\alpha>0$ is a scaling parameter and $\gamma\in[0,1)$ is the elliptic modulus.
In [8], it was shown that all such solutions
of the KdV are spectrally stable to localized perturbations by using the inverse scattering
transform.  It follows by Lemma \ref{perkdv:lem} that $\{T,M,P\}_{a,E,c}\geq 0$ in this case.
This result can be directly verified without the machinery of the inverse scattering transform 
by using standard elliptic function arguments.  In particular, noticing that in the case of the KdV
the quantities $T$, $M$, and $P$ can be represented by Abelian
integrals of the first, second, and third kinds, respectively, on a Riemann surface it can
be shown that the Jacobian $\{T,M,P\}_{a,E,c}$ can be expressed explicitly as
\[
\{T,M,P\}_{a,E,c}=\frac{T^3\left(E-V\left(M/T\right)\right)}{2\;{\rm disc}(R(u;a,E,c))^3},
\]
where $R(u;a,E,c):=E-V(u;a,c)$ and ${\rm disc}(R)$ denotes the discriminant of the polynomial $R$
(see [11] for details).
Since $R(u;a,E,c)$ clearly has three real roots for each $(a,E,c)\in\Omega$, and since $V(M/T)<E$,
it follows that $\{T,M,P\}_{a,E,c}>0$ for all $(a,E,c)\in\Omega$ in the case of the KdV
equation.  Therefore, Theorem \ref{transverse:ZK} implies the cnoidal solutions of the
KdV are spectrally unstable to long wavelength transverse
perturbations in the ZK equation.  We record this observation in the following corollary.


\begin{corr}
Periodic traveling wave solutions of the KdV with $(a_0,E_0,c_0)\in\Omega$
are spectrally unstable to long wavelength transverse perturbations in the ZK equation.
\end{corr}



Finally, we end by analyzing Theorem \ref{transverse:ZK} in the case of the focusing mKdV equation
\begin{equation}
u_t=u_{xxx}+u^2u_x.\label{mkdv}
\end{equation}
This equation admits two distinct families of solutions expressible in terms of Jacobi elliptic functions: up to Lie
symmetries, they are given by
\begin{align}
u(x,t)&=\sqrt{2}\alpha\dn(\mu x+\alpha^2(2-\gamma^2);\gamma),\label{mkdv:dn}\\
u(x,t)&=\sqrt{2}\alpha \gamma\cn(\mu x+\alpha^2(-1+2\gamma^2);\gamma)\label{mkdv:cn}.
\end{align}
In contrast to the case of the KdV, where the elliptic function solutions represented all periodic traveling
wave solutions, the solutions \eqref{mkdv:dn} and \eqref{mkdv:cn} only represent the
periodic traveling wave solutions of \eqref{mkdv} with $a=0$.  Stability calculations for
such solutions has been recently carried out in [11] and [12] in the context of nonlinear and spectral
stability to various classes of perturbations.  Our theory applies not only to these elliptic function solutions, but to
all periodic traveling wave solutions of the mKdV equation.
In [11], the signs of the Jacobians $\{T,M\}_{a,E}$ and $\{T,M,P\}_{a,E,c}$
were computed numerically for all $(a,E,+1)\in\Omega$ using the Abelian integral representation
of the necessary quantities.  In particular, it was shown that for all $(a,E,1)\in\Omega$ where $\{T,M\}_{a,E}$
is negative, the periodic stability index $\{T,M,P\}_{a,E,c}$ is also negative and hence one has
instability for all $(a,E,1)\in\Omega$.  We record this observation in the following corollary.

\begin{corr}\label{transversecorr2:ZK}
Periodic traveling waves of the focusing mKdV equation are unstable to
transverse perturbations in the focusing mZK equation.
\end{corr}

\section{Transverse Instabilities in the gBBM-ZK equation}

In this section, we consider the spectral instability of a periodic traveling wave solution of the 
gBBM equation to long wavelength transverse perturbations in the gBBM-ZK equation.  We begin by briefly
reviewing the basic properties of the periodic traveling waves of \eqref{gbbm} (see [18] for more details).
For each wave speed $c>1$, equation \eqref{gbbm} admits traveling wave solutions whose profiles
are solutions of the traveling wave ODE
\begin{equation}
cu_{xxt}-(c-1)u_x+f(u)_x=0,\label{travelode2}
\end{equation}
i.e. they are stationary solutions of \eqref{gbbm} in the traveling coordinate frame defined by $x-ct$.
Clearly \eqref{travelode2} defines a Hamiltonian ODE and can be reduced to quadrature: in particular,
the traveling wave profiles satisfy the relation
\[
\frac{c}{2}u_x^2-\left(\frac{c-1}{2}\right)u^2+F(u)=au+E
\]
where $a$ and $E$ are constants of integration and $F'=f$ with $F(0)=0$.  In order
to ensure the existence of periodic solutions of \eqref{travelode2}, we must
require that the effective potential
\[
G(u;a,c)=F(u)-\frac{c-1}{2}u^2-au
\]
have a non-degenerate local minimum.  As in the case of the gKdV equation, it follows that
the periodic solutions of \eqref{travelode2} comprise a four parameter
family of solutions $u(x+x_0;a,E,c)$ while the stationary solutions form a codimension two
subset.  Moreover, we mod out the translation invariance and hence only consider
the periodic solutions of \eqref{travelode2} as consisting of a three parameter family
indexed by the quantities $a$, $E$, and $c$.  

The partial differential equation \eqref{gbbm} has, in general, three conserved quantities
\begin{eqnarray}
M&=&\int_0^T\left(u-u_{xx}\right)dx\nonumber\\
P&=&\frac{1}{2}\int_0^T\left(u^2+u_x^2\right)dx\label{momentum}\\
H&=&\int_0^T\left(\frac{1}{2}u^2+F(u)\right)dx\nonumber
\end{eqnarray}
which correspond to the mass, momentum, and Hamiltonian (energy) of the solution, respectively.  These three
quantities are considered as functions of the traveling wave parameters $a$, $E$, and $c$ and their
gradients with respect to these parameters will play an important role in the foregoing analysis.  It
is important to notice that when restricted to the four-parameter family of periodic traveling wave
solutions of \eqref{gbbm}, the mass can be represented as $M=\int_0^Tu\;dx$.  Since all our results
concern this four-parameter family, we will always work with this simplified expression for the mass
functional.

We now begin our transverse instability analysis in this case.  Suppose we have a periodic traveling wave solution
$u(x;a,E,c)$ of the gBBM equation.  We wish to examine the spectral stability of $u$ to long wavelength 
transverse perturbations in the gBBM-ZK equation \eqref{gbbmzk}.  Clearly, $u$ is a $y$-independent
solution of the traveling gBBM-ZK equation
\begin{equation}
u_t-\left(u_{xt}+u_{yy}\right)_x-(c-1)u_x+f(u)_x+cu_{xxx}=0.\label{travelbbmzk}
\end{equation}
Linearizing \eqref{travelbbmzk} about $u$ yields the linear partial differential equation
\begin{equation*}
\mathcal{D}v_t=\partial_x\left(\mathcal{R}[u]+\partial_y^2\right)v,
\end{equation*}
where $\mathcal{D}=1-\partial_x^2$ and $\mathcal{R}[u]=-c\partial_x^2+(c-1)-f'(u)$.  
Seeking separated solutions of the form
\[
v(x,y,t)=v(x)e^{\mu t-iky}
\]
where $\mu\in\CM$ and $k\in\RM$ is the transverse wave number leads to the (ODE) spectral
problem
\[
\partial_x\left(\mathcal{R}[u]+k^2\right)v=\mu\mathcal{D}v
\]
considered on $L^2(\RM)$.  Notice that the operator $\mathcal{D}$ is positive and invertible on $L^2(\RM)$.
In particular, for each fixed $k\in\RM$ the spectrum of the operator $\mathcal{D}^{-1}\partial_x\left(\mathcal{R}[u]+k^2\right)$
is purely continuous and can be described via Floquet theory.  In particular, for a given $k\in\RM$
we have that $\mu\in{\rm spec}\left(\mathcal{D}^{-1}\partial_x\left(\mathcal{R}[u]+k^2\right)\right)$
if and only if there exists a $\lambda\in S^1$ such that
\[
D(\mu,k,\lambda):=\det\left(\MM(\mu,k)-\lambda {\bf I}\right)
\]
vanishes, where $\MM(\mu,k)$ is the monodromy map for the corresponding first order system
\begin{equation}
{\bf Y}_x=\HM(x,\mu,k){\bf Y},~~{\bf Y}(0,\mu,k)={\bf I},\label{bbmzksystem}
\end{equation}
where
\[
\HM(x,\mu,k)=\left(
               \begin{array}{ccc}
                 0 & 1 & 0 \\
                 0 & 0 & 1 \\
                 -\frac{1}{c}\left(\mu+f''(u)u_x\right) & \frac{1}{c}\left(c-1+k^2-f'(u)\right) & \frac{\mu}{c} \\
               \end{array}
             \right)
\]
We now state the main technical lemma for this section.

\begin{lem}\label{bbmzklem}
The equation $D(\mu,k,1)$ has the following local normal form in a neighborhood of the origin $(\mu,k)=(0,0)$:
\[
-\mu^3\{T,M,P\}_{a,E,c}+2\mu k^2\{T,M\}_{a,E}\int_0^Tu_x^2dx+\mathcal{O}\left(|\mu|^4+k^4\right).
\]
\end{lem}

\begin{proof}
In [18], it was shown that three linearly independent solutions of the equation $\mathcal{R}[u]v=0$
are given by the functions $u_x$, $u_a$, and $u_E$.  Letting $\WM(x,\mu,k)$ be a matrix
solution of \eqref{bbmzksystem} such that
\[
\WM(x,0,0)=\left(
             \begin{array}{ccc}
               cu_x & cu_{a} & cu_{E} \\
               cu_{xx} & cu_{ax} & cu_{Ex} \\
               cu_{xxx} & cu_{axx} & cu_{Exx} \\
             \end{array}
           \right)
\]
and defining $\delta\WM(\mu,k)=\WM(x,\mu,k)\big{|}_{x=0}^T$, it was shown in [18] that
\[
D(\mu,k,1)=-\frac{1}{c}\det\left(\delta\WM(\mu,k)\right).
\]
and that 
\[
D(\mu,0,1)=-\{T,M,P\}_{a,E,c}\mu^3+\mathcal{O}(|\mu|^4)
\]
for $|\mu|\ll 1$.  Moreover, the first order variations in $\mu$ of the vector solutions $Y_1$, $Y_2$, and $Y_3$
corresponding to the functions $u_x$, $u_a$, and $u_E$ at $\mu=0$ were calculated and essentially
have the same form as those found in the previous section for the gKdV equation.  Due to the particular
form of $\delta\WM(\mu,0)$, it follows then that one only needs to compute the first order variation
in $k^2$ of the vector solution $Y_1$ corresponding to the translation direction $u_x$.  Using
variation of parameters, one can show that
\[
\WM(T,0,0)\int_0^T\WM(z,0,0)^{-1}\left(
                                   \begin{array}{c}
                                     0 \\
                                     0 \\
                                     u_{xx}(z) \\
                                   \end{array}
                                 \right)dz
                                 =
                                 \left(
                                   \begin{array}{c}
                                     c\frac{\partial u(0)}{\partial E}\int_0^Tu_x^2dx \\
                                     * \\
                                     (c-1-f'(u(0)))\frac{\partial u(0)}{\partial E}\int_0^Tu_x^2dx \\
                                   \end{array}
                                 \right)
\]
where the term $*$ can be explicitly computed but is not necessary at this order in the perturbation
argument.  A straight forward calculation then gives
\[
\frac{1}{\mu k^2}\det\left(\delta\WM(\mu,k)\right)\big{|}_{(\mu,k)=(0,0)}=-c\{T,M\}_{a,E}\int_0^Tu_x^2dx.
\]
from which the lemma follows.
\end{proof}

Lemma \ref{bbmzklem} readily yields a necessary condition for the underlying (gBBM) periodic traveling wave
$u(x,a,E,c)$ to exhibit a modulational transverse instability in the gBBM-ZK equation.

\begin{thm}\label{bbmzkthm}
If $\{T,M,P\}_{a,E,c}\neq 0$,  then the spectrally stable periodic traveling wave solution of the gBBM equation
is spectrarlly unstable to long wavelength transverse perturbations in the 
gBBM-ZK equation if $\{T,M\}_{a,E}>0$.
\end{thm}

As an application of this theory, one can use the complex analytic methods of Bronski, Johnson, and Kapitula [11]
in order to prove that in the case of the BBM equation with $f(u)=u^2$, one has
\[
\{T,M\}_{a,E}=-\frac{T^3G'\left(\frac{M}{T}\right)}{12~{\rm disc}(E-G(~\cdot~;a,c))}
\]
where $G$ is the effective potential energy in this case.  In particular, since $G'$ is strictly
convex in the case of the BBM equation it follows from Jensen's inequality that
\[
G'\left(\frac{M}{T}\right)<\frac{1}{T}\int_0^TG'(u(x))dx=0
\]
from which it follows that $\{T,M\}_{a,E}>0$ for all such solutions.  In particular, this 
proves the following corollary of Theorem \ref{bbmzkthm}.

\begin{corr}
Let $u(x;a,E,c)$ be a periodic traveling wave solution of the BBM equation.  If $\{T,M,P\}_{a,E,c}$ is
non-zero, then $u$ is spectrally unstable to long wavelength transverse perturbations
in the BBM-ZK equation.
\end{corr}

Moreover, in [18] the Jacobians $\{T,M\}_{a,E}$ and $\{T,M,P\}_{a,E,c}$ were analyzed in the case of
power-law nonlinearities $f(u)=u^{p+1}/(p+1)$ in neighborhoods of the homoclinic orbits.  In particular, it was shown
that for all $p\neq 4$, the quantities $\{T,M,P\}_{a,E,c}$ and $\{T,M\}_{a,E,c}$ are both positive 
for all dnoidal type solutions of sufficiently long wavelength and $|a|$ sufficiently small.  This prove the following corollary
of Theorem \ref{bbmzkthm}.

\begin{corr}
In the case of power-law nonlinearities, the dnoidal-like periodic traveling wave solutions
of \eqref{gbbm} of sufficiently long wavelength and $|a|$ sufficiently small are spectrally
unstable to long wavelength transverse perturbations in the gBBM-ZK equation \eqref{gbbmzk}.
\end{corr}

\section{Discussion}

In this paper, we extended the methods of [10] to study the transverse instabilities of periodic traveling
wave solutions of the gKdV equation within a cannonical two-dimensional model: the
gZK equation. The results are sufficient conditions for the instability to
long wavelength transverse perturbations..  In particular, it follows from the results
in this paper that the cnoidal solutions of KdV are spectrally unstable to such perturbations
in both the ZK equation, and the dnoidal solutions of the focusing mKdV are unstable to such perturbations
in the focusing mZK equation.  
We then extended the results in [18] by outlining the corresponding transverse
instability theory for the gBBM equation within the gBBM-ZK equation yielding analogous results.

An obvious question arises as to the meaning of the Jacobian $\{T,M\}_{a,E}$ in each of these cases.  In this paper, we have seen
that it provides sufficient information for the transverse instability of periodic traveling wave solutions
of the gKdV and gBBM in several higher dimensional extensions.  In [11] and [17], this Jacobian was also shown to
arise naturally in the orbital stability of such solutions of the gKdV to perturbations with the same periodic structure as
the underlying periodic wave, and in [18] the analogous theory was outlined for the gBBM equation.
In particular, in [17] it was found that the quantity $\{T,M\}_{a,E}$ encodes information about the possible difference in the number
of negative periodic eigenvalues of the operators $\mathcal{L}[u]$ and $\mathcal{L}[u]\big{|}_{H_1}$, where
$H_1$ is the space of mean zero periodic functions: notice that by the form of the gKdV all non-trivial evolution of initial data occur within $H_1$.
It would be interesting to understand the exact physical information the Jacobian
$\{T,M\}_{a,E}$ encodes about a periodic traveling wave solution of gKdV and the gBBM.  The key to this
may lie in performing a formal Whitham theory calculation in the context of slow modulations of
a periodic traveling wave solution of the gKdV and comparing the formal result to the rigorous Whitham theory type calculatios
of this paper and 
[10] and [18].  We have not yet performed this calculation.

It seems clear that the methods employed in this paper can be applied in a straightforward way to studying
the transverse instabilities in other nonlinear PDE where the corresponding traveling wave ODE
is {\em close enough} to being completely integrable, in the sense that one can use the integrability
to construct all elements of the null-space of the linearized operator whether by Noether's theorem
or by variation of parameters, and where the corresponding
linearization yields a spectral problem for a linearized operator.  
It is not clear however if they can be applied to equations such as the gKP equation
\[
\left(u_t-u_{xxx}-f(u)_x\right)+\sigma u_{yy}=0
\]
which arises as a weakly two-dimensional variation of the gKdV equation.  In particular, the corresponding
spectral problem arising from linearizing about a periodic traveling wave of the gKdV takes the form
\[
\left(\partial_x^2\mathcal{L}[u]+\sigma k^2\right)v=\mu\partial_x v
\]
and hence can not be realized as a spectral problem for any particular linear operator since $\partial_x$
is not invertible on $L^2_{\rm per}([0,T])$.  It would be very interesting to extend the Evans function techniques
in this paper to consider equations of this form.


\end{document}